\documentclass[11pt, reqno]{amsart}
\usepackage{amssymb, amstext, amscd, amsmath, amsthm}
\usepackage{color}
\usepackage{dsfont}

\usepackage[notcite,notref]{showkeys}   
\usepackage{hyperref}

\usepackage{xy}
\xyoption{all}

\theoremstyle{plain}
\newtheorem{theorem}{Theorem}[section]
\newtheorem{thm}[theorem]{Theorem}

\newtheorem{prop}[theorem]{Proposition}

\newtheorem*{theorem*}{Theorem}
\theoremstyle{definition}

\newtheorem{defn}[theorem]{Definition}

\newtheorem{notation}[theorem]{Notation}

\renewcommand{\O}{{\mathcal{O}}}

\newcommand{\upchi}{{\raise.35ex\hbox{\ensuremath{\chi}}}}

\newcommand{\spn}{\operatorname{span}}

\newcommand{\ca}{\mathrm{C}^*}

\newcommand{\mt}{\varnothing}

\newcommand{\Per}{{\rm Per}}

\begin{document}
\title[]{The Jacobson Topology of The Primitive Ideal Space of Self-Similar $k$-Graph C*-Algebras}
\author[]{Hui Li}
\address{Hui Li, Department of Mathematics and Physics, North China Electric Power University, Beijing 102206, China}
\email{lihui8605@hotmail.com}

\thanks{The author was supported by National Natural Science Foundation of China (Grant No.~11801176) and by Fundamental Research Funds for the Central Universities (Grant No.~2020MS040).}

\begin{abstract}
We describe the Jacobson topology of the primitive ideal space of self-similar $k$-graph C*-algebras under certain conditions.
\end{abstract}

\subjclass[2010]{46L05}
\keywords{C*-algebra, self-similar $k$-graph, primitive ideal, Jacobson topology}

\maketitle



\section{Self-similar $k$-graph C*-algebras}







In this subsection, we recall the background of $k$-graph C*-algebras and self-similar $k$-graph C*-algebras from \cite{CKSS14, KP00, LY17, LY19}.

\begin{defn}
Let $k$ be a positive integer. A countable small category $\Lambda$ is called a \emph{$k$-graph} if there exists a functor $d:\Lambda \to \mathbb{N}^k$ satisfying that for $\mu\in\Lambda, p,q \in \mathbb{N}^k$ with $d(\mu)=p+q$, there exist unique $\alpha,\beta \in \Lambda$ such that $d(\alpha)=p,d(\beta)=q,s(\alpha)=r(\beta),\mu=\alpha\beta$. The functor $d$ is called the \emph{degree map}. Moreover, a functor between two $k$-graphs is called a \emph{graph morphism} if it preserves the degree maps.
\end{defn}

\begin{notation}
Let $k$ be a positive integer. Define $\Omega_k:=\{(p,q) \in \mathbb{N}^k \times \mathbb{N}^k:p \leq q\}$; define $\Omega_k^0:=\{(p,p):p \in \mathbb{N}^k\}$; for $(p,q), (q,m) \in \Omega_k$, define $r(p,q):=(p,p); s(p,q):=(q,q);(p,q) \cdot (q,m):=(p,m); d(p,q):=q-p$. Then $\Omega_k$ is a $k$-graph.
\end{notation}

\begin{notation}
Let $\Lambda$ be a $k$-graph. For $A,B \subset \Lambda$, denote by $AB:=\{\mu\nu:\mu \in A,\nu\in B,s(\mu)=r(\nu)\}$. For $p \in \mathbb{N}^k$, denote by $\Lambda^p:=d^{-1}(p)$. 
\end{notation}

\begin{defn}
Let $\Lambda$ be a $k$-graph. A graph morphism from $\Omega_k$ to $\Lambda$ is called an \emph{infinite path}. The set of all infinite paths of $\Lambda$ is denoted by $\Lambda^\infty$. Fix $x \in \Lambda^\infty,n \in \mathbb{N}^k,\mu \in \Lambda x(0,0),(p,q) \in \Omega_k$. Denote by $\sigma^n(x),\mu x\in \Lambda^\infty$ such that $\sigma^n(x)(p,q)=x(p+n,q+n)$ and $(\mu x)(0,d(\mu)+n)=\mu x(0,n)$. Moreover, $x$ is said to be \emph{cofinal} if for any $v \in \Lambda^0$, there exists $m \in \mathbb{N}^k$, such that $v\Lambda x(m,m) \neq \emptyset$.
\end{defn}

\begin{defn}
Let $\Lambda$ be a $k$-graph. Then $\Lambda$ is said to be \emph{row-finite} if $\vert v\Lambda^{p}\vert<\infty$ for all $v \in \Lambda^0$ and $p \in \mathbb{N}^k$. $\Lambda$ is said to be \emph{source-free} if $v\Lambda^{p} \neq \mt$ for all $v \in \Lambda^0$ and $p \in \mathbb{N}^k$.
\end{defn}

\subsection*{Standing Assumptions:} \textsf{Throughout the rest of this paper, all $k$-graphs are assumed to be row-finite and source-free.}

\begin{defn}
Let $\Lambda$ be a $k$-graph. Define the \emph{$k$-graph C*-algebra} $\O_\Lambda$ to be the universal C*-algebra generated by a family of partial isometries $\{s_\lambda:\lambda\in\Lambda\}$ (\emph{Cuntz-Krieger $\Lambda$-family}) satisfying
\begin{enumerate}
\item $\{s_v\}_{v \in \Lambda^0}$ is a family of mutually orthogonal projections;
\item $s_{\mu\nu}=s_{\mu} s_{\nu}$ if $s(\mu)=r(\nu)$;
\item $s_{\mu}^* s_{\mu}=s_{s(\mu)}$ for all $\mu \in \Lambda$; and
\item $s_v=\sum_{\mu \in v \Lambda^{p}}s_\mu s_\mu^*$ for all $v \in \Lambda^0, p \in \mathbb{N}^k$.
\end{enumerate}
\end{defn}

\begin{defn}
Let $\Lambda$ be a $k$-graph. Let $H$ be a subset of $\Lambda^0$. Then $H$ is said to be \emph{hereditary} and \emph{saturated} if
\begin{enumerate}
\item $s(H \Lambda) \subset H$;
\item for any $v \in \Lambda^0$, if there exists $p \in \mathbb{N}^k$ satisfying $s(v \Lambda^p) \subset H$, then $v \in H$.
\end{enumerate}
For any subset $A$ of $\Lambda^0$, denote by $\sum A$ the smallest hereditary and saturated set containing $A$. Let $T$ be a nonempty subset of $\Lambda^0$. Then $T$ is called a \emph{maximal tail} if
\begin{enumerate}
\item for any $v \in T,w \in \Lambda^0$, we have $w \Lambda v \neq \emptyset \implies w \in T$;
\item for any $v \in T, p \in \mathbb{N}^k$, we have $v \Lambda^p T \neq \emptyset$;
\item for any $v_1,v_2 \in T$, there exists $w \in T$ such that $v_1 \Lambda w,v_2 \Lambda w \neq \emptyset$.
\end{enumerate}
\end{defn}

\begin{defn}
Let $\Lambda$ be a $k$-graph, let $G$ be a countable discrete group acting on $\Lambda$, and let $\vert:G\times \Lambda \to G$ be a map. Then $(G,\Lambda)$ is called a \emph{self-similar $k$-graph} if
\begin{enumerate}
\item $G \cdot \Lambda^p \subset \Lambda^p$ for all $p \in \mathbb{N}^k$;
\item $s(g \cdot \mu)=g \cdot s(\mu)$ and $r(g \cdot \mu)=g \cdot r(\mu)$ for all $g \in G,\mu \in \Lambda$;
\item
$g\cdot (\mu\nu)=(g \cdot \mu)(g \vert_\mu \cdot \nu)$ for all $g \in G,\mu,\nu \in \Lambda$ with $s(\mu)=r(\nu)$;

\item
$g \vert_v =g$ for all $g \in G,v \in \Lambda^0$;

\item
$g \vert_{\mu\nu}=g \vert_\mu \vert_\nu$ for all $g \in G,\mu,\nu \in \Lambda$ with $s(\mu)=r(\nu)$;

\item
$1_G \vert_{\mu}=1_G$ for all $\mu \in \Lambda$;

\item
$(gh)\vert_\mu=g \vert_{h \cdot \mu} h \vert_\mu$ for all $g,h \in G,\mu \in \Lambda$.
\end{enumerate}
Furthermore, $(G,\Lambda)$ is said to be \emph{pseudo free} if for any $g \in G,\mu \in \Lambda,g \cdot \mu=\mu,g \vert_\mu=1_G \implies g=1_G$.
\end{defn}

\begin{defn}
Let $(G,\Lambda)$ be a self-similar $k$-graph. Define $\mathcal{O}_{G,\Lambda}^\dagger$ to be the universal unital C*-algebra generated by a Cuntz-Krieger $\Lambda$-family $\{s_\mu:\mu\in\Lambda\}$ and a family of unitaries $\{u_g:g \in G\}$ satisfying
\begin{enumerate}
\item $u_{gh}=u_g u_h$ for all $g,h \in G$;
\item\label{u_g s_mu} $u_g s_\mu=s_{g \cdot \mu} u_{g \vert_\mu}$ for all $g \in G,\mu \in \Lambda$.
\end{enumerate}
Define $\mathcal{O}_{G,\Lambda}:=\overline{\spn}\{s_\mu u_g s_\nu^*:s(\mu)=g \cdot s(\nu)\}$, which is called the \emph{self-similar $k$-graph C*-algebra} of $(G,\Lambda)$.
\end{defn}

\begin{defn}
Let $(G,\Lambda)$ be a self-similar $k$-graph. For any $\mu,\nu \in \Lambda,g \in G$ with $s(\mu)=g \cdot s( \nu)$, if $\mu(g \cdot x)=\nu x$ for all $x \in s(\nu)\Lambda^\infty$, then $(\mu,g,\nu)$ is called a \emph{cycline triple}. Define $\mathrm{Per}_{G,\Lambda}:=\{d(\mu)-d(\nu):(\mu,g,\nu) \text{ is a cycline}$ $\text{triple}\}$. Cycline triples of the form $(\mu,1_G,\nu)$ are simply called \emph{cycline pairs}. Define $\mathrm{Per}_{\Lambda}:=\{d(\mu)-d(\nu):(\mu,1_G,\nu) \text{ is a cycline pair}\}$. Denote by $M(\Lambda)$ the set of all maximal tails of $\Lambda$, by $M_\gamma(\Lambda):=\{T \in M(\Lambda):\Per_{ \Lambda T}=0\}$, and by $M_\tau(\Lambda):=\{T \in M(\Lambda):\Per_{ \Lambda T}\neq 0\}$.
\end{defn}

\begin{prop}
Let $(G,\Lambda)$ be a pseudo free self-similar $k$-graph such that $g\cdot v=v$ for all $g\in G,v\in \Lambda^0$. Let $T$ be a maximal tail of $\Lambda$. Then
\begin{enumerate}
\item there exists a nonempty hereditary subset of $H_T$ of $\Lambda T$ consisting of $v \in T$ satisfying for any $p,q \in \mathbb{N}^k$ with $p-q \in \Per_{\Lambda T}$, for any $\mu \in v\Lambda^pT$, there exists a unique $\nu \in v\Lambda^qT$ such that $(\mu,1_G,\nu)$ is a cycline pair of $(G,\Lambda T)$;
\item $(G,\Lambda T), (G,H_T \Lambda T)$ are pseudo free self-similar $k$-graphs;
\item if every cycline triple of $(G,\Lambda T)$ is a cycline pair of $(G,\Lambda T)$, then every cycline triple of $(G,H_T \Lambda T)$ is a cycline pair of $(G,H_T \Lambda T)$ and $\Per_{G,H_T \Lambda T}=\Per_{H_T \Lambda T}=\Per_{G, \Lambda T}=\Per_{\Lambda T}$ is a subgroup of $\mathbb{Z}^k$.
\end{enumerate}
\end{prop}


\begin{thm}\label{characterization of prim ideal of O_G,Lambda}
Let $(G,\Lambda)$ be a pseudo free self-similar $k$-graph satisfying
\begin{enumerate}
\item $g\cdot v=v$ for all $g\in G,v\in \Lambda^0$;
\item every cycline triple of $(G,\Lambda T)$ is a cycline pair of $(G,\Lambda T)$ for any maximal tail $T$ of $\Lambda$.
\end{enumerate}
Then there exists a bijection from $\amalg_{T \in M(\Lambda)} \{T\} \times \widehat{\Per}_{\Lambda T}$ onto $\mathrm{Prim}(\mathcal{O}_{G,\Lambda})$. For any maximal tail $T$ of $\Lambda$ and for any $f \in \widehat{\Per}_{\Lambda T}$, denote by $I_{T,f}$ the image of $(T,f)$. Find an arbitrary cofinal infinite path of $\Lambda T$ and an arbitrary $\widetilde{f} \in \mathbb{Z}^k$ extending $f$, denote by $[x]:=\{y \in (\Lambda T)^\infty:\exists \ p,q \in \mathbb{N}^k,g \in G,s.t. \ \sigma^p(x)=g \cdot \sigma^q(y)\}$. Then there is an irreducible representation $\pi_{\widetilde{f},x}:\mathcal{O}_{G,\Lambda} \to B(l^2([x]))$ such that for any $\mu \in \Lambda,g \in G,y \in [x]$, we have $\pi_{\widetilde{f},x}(s_\mu)\delta_y=\begin{cases}
   \widetilde{f}(d(\mu))\delta_{\mu y}   &\text{ if $s(\mu)=y(0,0)$} \\
   0   &\text{ otherwise }\end{cases},\pi(u_g)\delta_y=\delta_{g \cdot y}$. Moreover, $I_{T,f}=\ker(\pi_{\widetilde{f},x})$. In particular, $s_v \in I_{T,f}$ if and only if $v \notin T$.
\end{thm}

\section{The Jacobson Topology}

In this section, we describe the Jacobson topology of the primitive ideal space of self-similar $k$-graph C*-algebras under certain conditions. Our approach is inspired by \cite{HS04} and \cite{MR3837593}.

\begin{thm}\label{Jacob top}
Let $(G,\Lambda)$ be a pseudo free self-similar $k$-graph satisfying
\begin{enumerate}
\item\label{sss} $g\cdot v=v$ for all $g\in G,v\in \Lambda^0$;
\item every cycline triple of $(G,\Lambda T)$ is a cycline pair of $(G,\Lambda T)$ for any maximal tail $T$ of $\Lambda$;
\item for any $T \in M_\tau(\Lambda),H_T \Lambda T$ is strongly connected;
\item for any maximal tail $T$ of $\Lambda$ and for any $f \in \widehat{\Per}_{G,\Lambda T}, I_{T,f}$ is the closed two-sided ideal of $\mathcal{O}_{G,\Lambda}$ generated by $\{s_v\}_{v \notin T}$ and $\{s_{\mu}-f(d(\mu)-d(\nu))s_\nu:(\mu,1_G,\nu) \text{ is a cycline pair of }(G,H_{T}\Lambda T)\}$.
\end{enumerate}
Let $W$ be a nonempty subset of $M_\gamma(\Lambda)$, let $Y$ be a nonempty subset of $M_\tau(\Lambda)$, let $D(T)$ be a nonempty subset of $\widehat{\Per}_{G, \Lambda T}$ for each $T \in Y$, let $S_0 \in M_\gamma(\Lambda)$, let $T_0 \in M_\tau(\Lambda)$, and let $f_0 \in \widehat{\Per}_{G,\Lambda T_0}$. Then
\begin{enumerate}
\item\label{S in W} $I_{S_0,1} \in \overline{\{I_{S,1}:S \in W\}}$ if and only if $S_0 \subset \bigcup_{S \in W}S$;
\item\label{T in W} $I_{T_0,f_0} \in \overline{\{I_{S,1}:S \in W\}}$ if and only if $T_0 \subset \bigcup_{S \in W}S$;
\item\label{S in Y} $I_{S_0,1} \in \overline{\{I_{T,f}:T \in Y,f \in D(T)\}}$ if and only if $S_0 \subset \bigcup_{T \in Y}T$;
\item\label{T in Y} $I_{T_0,f_0} \in \overline{\{I_{T,f}:T \in Y,f \in D(T)\}}$ if and only if either there exists $T_0 \neq T \in Y$ such that $T_0 \subset T$ or there exist no $T_0 \neq T \in Y$ such that $T_0 \subset T$,$T_0 \in Y$, $f_0 \in \overline{D(T_0)}$.
\end{enumerate}
\end{thm}
\begin{proof}
(\ref{S in W}). By \cite[Definition~A.19]{RW98}, $I_{S_0,1} \in \overline{\{I_{S,1}:S \in W\}}$ if and only if $\bigcap_{S \in W}I_{S,1}\subset I_{S_0,1}$. By \cite[Corollary~5.6]{LY19}, $I_{S_0,1}=I(\Lambda^0 \setminus S_0)$ which is the closed two-sided ideal of $\mathcal{O}_{G,\Lambda}$ generated by $\{s_v\}_{v \notin S_0}$. By \cite[Lemma~4.3, Theorem~4.5]{LY19}, $\bigcap_{S \in W}I_{S,1}=I(\Lambda^0 \setminus \bigcup_{S \in W}S)$. So $\bigcap_{S \in W}I_{S,1}\subset I_{S_0,1}$ if and only if $S_0 \subset \bigcup_{S \in W}S$.

(\ref{T in W}). $I_{T_0,f_0} \in \overline{\{I_{S,1}:S \in W\}}$ if and only if $I(\Lambda^0 \setminus \bigcup_{S \in W}S) \subset I_{T_0,f_0}$. Suppose that $I(\Lambda^0 \setminus \bigcup_{S \in W}S) \subset I_{T_0,f_0}$. By Theorem~\ref{characterization of prim ideal of O_G,Lambda}, $T_0 \subset \bigcup_{S \in W}S$. Conversely, suppose that $T_0 \subset \bigcup_{S \in W}S$. Then $I(\Lambda^0 \setminus \bigcup_{S \in W}S) \subset I(\Lambda^0 \setminus T_0)$. By the fourth assumption of this theorem $I(\Lambda^0 \setminus \bigcup_{S \in W}S) \subset I_{T_0,f_0}$.

(\ref{S in Y}). $I_{S_0,1} \in \overline{\{I_{T,f}:T \in Y,f \in D(T)\}}$ if and only if $\bigcap_{T \in Y,f \in D(T)}I_{T,f} \subset I_{S_0,1}$.

Suppose that $\bigcap_{T \in Y,f \in D(T)}I_{T,f} \subset I_{S_0,1}$. By the fourth assumption of this theorem $I(\Lambda^0 \setminus \bigcup_{T \in Y}T)=\bigcap_{T \in Y}I(\Lambda^0 \setminus T)\subset I_{S_0,1}$. By Theorem~\ref{characterization of prim ideal of O_G,Lambda}, $\Lambda^0 \setminus \bigcup_{T \in Y}T \subset \Lambda^0 \setminus S_0$. So $S_0 \subset  \bigcup_{T \in Y}T$.

Conversely, suppose that $S_0 \subset \bigcup_{T \in Y}T$. Denote by $\pi:\mathcal{O}_{G,\Lambda} \to \mathcal{O}_{G,\Lambda}/I(\Lambda^0$ $\setminus S_0) \cong \mathcal{O}_{G,\Lambda S_0}$ the quotient map. In order to show that $\bigcap_{T \in Y,f \in D(T)}I_{T,f} \subset I_{S_0,1}$, we only need to prove that $\pi(\bigcap_{T \in Y,f \in D(T)}I_{T,f})=0$. Fix $v_0 \in S_0$. Suppose that $\pi(s_{v_0}) \in \pi(\bigcap_{T \in Y,f \in D(T)}I_{T,f})$, for a contradiction. Then there exists $a \in I(\Lambda^0 \setminus S_0)$ such that $s_{v_0} -a \in \bigcap_{T \in Y,f \in D(T)}I_{T,f}$. So there exist $\{\mu_i,\nu_i\}_{i=1}^{n} \subset \Lambda (\Lambda^0 \setminus S_0)$ and $\{g_i\}_{i=1}^{n} \subset G$ such that $\Vert a-\sum_{i=1}^{n}s_{\mu_i} u_{g_i} s_{\nu_i}^*\Vert<1/2$. Let $p:=\sum_{i=1}^{n}(d(\mu_i)+d(\nu_i))+\sum_{i=1}^{k}e_i$. Since $v_0 \in S_0$ and $S_0$ is a maximal tail, there exists $\xi \in v_0 \Lambda^p S_0$. Since $S_0 \subset \bigcup_{T \in Y}T$, there exist $T_1 \in Y,f_1 \in D(T_1)$ such that $s(\xi) \in T_1$. Fix an arbitrary cofinal infinite path $x$ in $\Lambda T_1$ and fix $\widetilde{f_1} \in \widehat{\mathbb{Z}}^k$ extending $f_1$. Then there exists $q \in \mathbb{N}^k$ such that $\eta \in s(\xi) \Lambda x(q,q)$. We calculate that
\begin{align*}
\Vert \pi_{\widetilde{f_1},x}(s_{v_0}-a) \Vert&=\Vert \pi_{\widetilde{f_1},x}(s_{v_0}-\sum_{i=1}^{n}s_{\mu_i} u_{g_i} s_{\nu_i}^*)+\pi_{\widetilde{f_1},x}(\sum_{i=1}^{n}s_{\mu_i} u_{g_i} s_{\nu_i}^*-a) \Vert
\\&\geq \Vert \pi_{\widetilde{f_1},x}(s_{v_0}-\sum_{i=1}^{n}s_{\mu_i} u_{g_i} s_{\nu_i}^*)\Vert-\Vert\pi_{\widetilde{f_1},x}(\sum_{i=1}^{n}s_{\mu_i} u_{g_i} s_{\nu_i}^*-a) \Vert
\\&>\Vert \pi_{\widetilde{f_1},x}(s_{v_0}-\sum_{i=1}^{n}s_{\mu_i} u_{g_i} s_{\nu_i}^*)\Vert-1/2
\\&\geq \Vert \pi_{\widetilde{f_1},x}(s_{v_0}-\sum_{i=1}^{n}s_{\mu_i} u_{g_i} s_{\nu_i}^*)\delta_{\xi\eta \sigma^q(x)}\Vert-1/2
\\&=\Vert \delta_{\xi\eta \sigma^q(x)}\Vert-1/2
\\&=1/2.
\end{align*}
On the other hand, $s_{v_0}-a \in I_{T_1,f_1}$ because $s_{v_0}-a \in \bigcap_{T \in Y,f \in D(T)}I_{T,f}$. So $\pi_{\widetilde{f_1},x}(s_{v_0}-a)=0$, which is a contradiction. Hence $\pi(s_{v_0}) \notin \pi(\bigcap_{T \in Y,f \in D(T)}$ $I_{T,f})$. Denote by $q:\mathcal{O}_{G,\Lambda S_0} \to \mathcal{O}_{G,\Lambda S_0}/\pi(\bigcap_{T \in Y,f \in D(T)}I_{T,f})$ the quotient map. Then $\pi(s_v)+\pi(\bigcap_{T \in Y,f \in D(T)}I_{T,f})\neq \pi(\bigcap_{T \in Y,f \in D(T)}I_{T,f})$ for all $v \in S_0$. Since $S_0 \in M_\gamma(\Lambda)$, by \cite[Theorem~3.19]{LY19} $q$ is injective. So $\pi(\bigcap_{T \in Y,f \in D(T)}I_{T,f})=0$.

(\ref{T in Y}). $I_{T_0,f_0}$ is in the closure of $\{I_{T,f}:T \in Y,f \in D(T)\}$ if and only if $\bigcap_{T \in Y,f \in D(T)}I_{T,f} \subset I_{T_0,f_0}$.

Suppose that $\bigcap_{T \in Y,f \in D(T)}I_{T,f} \subset I_{T_0,f_0}$. By the fourth assumption of this theorem, $I(\Lambda^0 \setminus \bigcup_{T \in Y}T)=\bigcap_{T \in Y}I(\Lambda^0 \setminus T)\subset I_{T_0,f_0}$. By Theorem~\ref{characterization of prim ideal of O_G,Lambda}, $T_0 \subset \bigcup_{T \in Y}T$. So $H_{T_0} \subset \bigcup_{T \in Y}T$. We deduce that there exists $T \in Y$ such that $T_0 \subset T$. Now we split into two cases.

Case 1. There exists $T_0 \neq T \in Y$ such that $T_0 \subset T$.

Case 2. There exist no $T_0 \neq T \in Y$ such that $T_0 \subset T$. Then $T_0 \in Y$. By \cite[Proposition~A.17]{RW98}, either $\bigcap_{f \in D(T_0)}I_{T_0,f} \subset I_{T_0,f_0}$ or $\bigcap_{T_0 \neq T \in Y,f \in D(T)}I_{T,f} \subset I_{T_0,f_0}$. Suppose that $\bigcap_{T_0 \neq T \in Y,f \in D(T)}I_{T,f} \subset I_{T_0,f_0}$, for a contradiction. By the fourth assumption of this theorem, $I(\Lambda^0 \setminus \bigcup_{T_0 \neq T \in Y}T)\subset I_{T_0,f_0}$. By Theorem~\ref{characterization of prim ideal of O_G,Lambda}, $T_0 \subset \bigcup_{T_0 \neq T \in Y}T$. So $H_{T_0} \subset \bigcup_{T_0 \neq T \in Y}T$. We deduce that there exists $T_0 \neq T \in Y$ such that $T_0 \subset T$, which is a contradiction. Hence $\bigcap_{f \in D(T_0)}I_{T_0,f} \subset I_{T_0,f_0}$. By \cite[Corollary~5.6]{LY19}, $f_0 \in \overline{D(T_0)}$.

Conversely, suppose that there exists $T_0 \neq T \in Y$ such that $T_0 \subset T$. Then $H_T \cap T_0 =\emptyset$. So $\Lambda^0 \setminus T, H_T \subset \Lambda^0 \setminus T_0$. By the fourth assumption of this theorem, $I_{T,f} \subset I_{T_0,f_0}$.

Finally, suppose that there exist no $T_0 \neq T \in Y$ such that $T_0 \subset T$, $T_0 \in Y$, $f_0 \in \overline{D(T_0)}$. By \cite[Corollary~5.6]{LY19}, $\bigcap_{f \in D(T_0)}I_{T_0,f} \subset I_{T_0,f_0}$. So $\bigcap_{T \in Y,f \in D(T)}I_{T,f} \subset I_{T_0,f_0}$.
\end{proof}

\begin{prop}\label{prim ideal 1-graph}
Let $(G,\Lambda)$ be a pseudo free self-similar $1$-graph satisfying the first two conditions of Theorem~\ref{Jacob top}. Then the third and the fourth assumptions of Theorem~\ref{Jacob top} automatically hold.
\end{prop}
\begin{proof}
Pointed on \cite[Page~49]{HS04}, for any $T \in M_\tau(\Lambda),H_T$ is the vertices of the unique cycle without entrances of $\Lambda T$. So $H_T \Lambda T$ is the unique cycle without entrances of $\Lambda T$, which is strongly connected.

Now we verify the fourth assumption of Theorem~\ref{Jacob top} holds. Fix $T \in M_\tau(\Lambda)$, fix $f \in \widehat{\Per}_{G,\Lambda T}$. We consider the quotient $\mathcal{O}_{G,\Lambda T} \cong \mathcal{O}_{G,\Lambda}/I(\Lambda^0 \setminus T)$ and we regard $I_{T,f},I(H_T)$ as the closed two-sided ideals of $\mathcal{O}_{G,\Lambda T}$. It suffices to show that $I_{T,f}$ is the closed two-sided ideal of $\mathcal{O}_{G,\Lambda T}$ generated by $\{s_{\mu}-f(d(\mu)-d(\nu))s_\nu:(\mu,1_G,\nu) \text{ is a cycline pair of }(G,H_{T}\Lambda T)\}$. Denote by $J_{T,f}$ the closed two-sided ideal of $\mathcal{O}_{G,\Lambda T}$ generated by $\{s_{\mu}-f(d(\mu)-d(\nu))s_\nu:(\mu,1_G,\nu) \text{ is a cycline pair of }(G,H_{T}\Lambda T)\}$. As in the proof of \cite[Proposition~5.4]{LY19}, $J_{T,f} \subset I_{T,f}$. Denote by $q: \mathcal{O}_{G,\Lambda T}/J_{T,f} \to \mathcal{O}_{G,\Lambda T}/I_{T,f}$ the quotient map. We may assume that $H_T =\{v_0\}$ and that $H_T \Lambda^1 T=\{e_0\}$ (This is to simplify the notation, when $H_T$ has multiple vertices then the proof shares the similar argument).

Define $\mathcal{A}^\dagger$ to be the universal C*-algebra generated by a family of partial isometries $\{s_e\}_{e \in \Lambda^1 T \setminus \{e_0\}}$, a family of mutually orthogonal projections $\{s_v\}_{v \in T}$, and a family of unitaries $\{u_g\}_{g \in G}$, satisfying
\begin{enumerate}
\item $s_e^*s_e=s_{s(e)}$ for all $ e \in \Lambda^1 T \setminus \{e_0\}$;
\item $s_v=\sum_{e \in \Lambda^1 T}s_e s_e^*$ for all $v_0 \neq v \in T$;
\item $u_{gh}=u_g u_h$ for all $g,h \in G$;
\item $u_g s_e=s_{g \cdot e} u_{g \vert_e}$ for all $g \in G,e \in \Lambda^1 T \setminus \{e_0\}$;
\item $u_g s_v=s_v u_g$ for all $g \in G,v\in T$.
\end{enumerate}
Define $\mathcal{A}:=\overline{\spn}\{s_\mu u_g s_\nu^*:\mu,\nu \in \Lambda T \setminus \Lambda e_0,g \in G\}$. It is straightforward to see that there exists a natural surjective homomorphism $\rho:\mathcal{A} \to  \mathcal{O}_{G,\Lambda T}/J_{T,f}$.

Define $E^0:=T \amalg \{v_n\}_{n=1}^{\infty}$; define $E^1:=\Lambda^1 T \setminus \{e_0\} \amalg \{e_n\}_{n=1}^{\infty}$; define $r_E:E^1 \to E^0$ by $r_E \vert_{\Lambda^1}:=r$ and by $r_E(e_n):=v_{n-1}$ for all $n \geq 1$; and define $s_E:E^1 \to E^0$ by $s_E \vert_{\Lambda^1}:=s$ and by $s_E(e_n):=v_{n}$ for all $n \geq 1$. For $g \in G, n \geq 1$, define $g \cdot e_n:=e_n$; define $g \cdot v_n:=v_n$; define $g \vert_{e_n}:=g$; and define $g \vert_{v_n}:=g$. By \cite[Proposition~2.1]{EP17}, $(G,E)$ is a pseudo-free self-similar $1$-graph. Let $P:=\sum_{v \in T}s_v$ which is a projection of $M(\mathcal{O}_{G,E})$. It is easy to check that $P \mathcal{O}_{G,E} P$ is a full corner of $\mathcal{O}_{G,E}$. By \cite[Theorem~3.2]{EPS18}, there exists a natural isomorphism $\iota:\mathcal{A} \to P \mathcal{O}_{G,E} P$.

Take an arbitrary nondegenerate injective representation $\pi:\mathcal{O}_{G,\Lambda T}/I_{T,f} \to B(H)$. Define $\varphi:=\pi \circ q \circ \rho \circ \iota^{-1}$. Since $(G,E)$ is aperiodic, by \cite[Proposition~2.66]{RW98} and by \cite[Theorem~3.19]{LY19} the induced representation $Ind-\varphi$ is injective on $ \mathcal{O}_{G,E} $. So $\varphi$ is injective. Hence $q$ is injective. Therefore $I_{T,f}=J_{T,f}$.
\end{proof}

\section*{Acknowledgments}

The author wants to thank Professor Toke Carlsen and Professor Aidan Sims for valuable discussions on both Theorem~\ref{Jacob top} and Proposition~\ref{prim ideal 1-graph}. The author especially wants to thank Professor Aidan Sims for numerous patient email correspondences.

\end{document}